\documentclass[11pt,reqno]{amsart}
\usepackage[a4paper, hmargin={2.7cm,2.7cm},vmargin={3.3cm,3.3cm}]{geometry}
\usepackage[english]{babel}
\usepackage{color}
\usepackage[noadjust]{cite}
\usepackage{amsmath}
\usepackage{amssymb}
\usepackage{amsfonts}
\usepackage{amsthm}
\usepackage{enumerate}
\usepackage{amsthm}
\usepackage{dsfont}

\usepackage[textwidth=20mm]{todonotes}
\usepackage[hypertexnames=false]{hyperref}
\usepackage[nameinlink]{cleveref}
\usepackage{commath}

\usepackage{etoolbox}

\usepackage{amssymb}   
\usepackage{amsthm}    
\usepackage{thmtools}  
\usepackage{mathtools} 
\usepackage{mathrsfs}  
\usepackage{commath}
\usepackage{bbm}
\usepackage[textwidth=20mm]{todonotes}

\theoremstyle{plain}
\newtheorem{theorem}{Theorem}[section]
\newtheorem{lemma}[theorem]{Lemma}

\theoremstyle{definition}

\newtheorem{examplex}[theorem]{Example}
\newenvironment{example}
  {\pushQED{\qed}\examplex}
  {\popQED\endexamplex}

\theoremstyle{remark}
\newtheorem{remark}[theorem]{Remark}

\usepackage[all]{xy}

\numberwithin{equation}{section}

\newcommand{\Hpi}{\mathcal{H}_{\pi}}

\DeclareMathOperator{\covol}{covol}

\DeclareMathOperator{\Span}{span}

\newcommand{\Ad}{{\rm Ad}}
\newcommand{\Aut}{{\rm Aut}\,}
\newcommand{\ad}{{\rm ad}}

\newcommand{\ee}{{\rm e}}

\newcommand{\Hc}{\mathcal H}

\title{Coherent frames with zero Beurling density}

\author{Ingrid Belti\c t\u a}
\address{Institute of Mathematics "Simion Stoilow" of the Romanian Academy, PO Box 1-764, 014700 Bucharest, Romania}
\email{ingrid.beltita@gmail.com, Ingrid.Beltita@imar.ro}

\author{Jordy Timo van Velthoven}
\address{Faculty of Mathematics,
University of Vienna, 
Oskar-Morgenstern-Platz 1,
1090 Vienna, Austria}
\email{jordy-timo.van-velthoven@univie.ac.at}

\makeatletter
\@namedef{subjclassname@2020}{\textup{2020} Mathematics Subject Classification}
\makeatother



\begin{document}

\begin{abstract}
We show the existence of a coherent frame in the orbit of a connected, simply connected unimodular solvable Lie group of exponential growth for which the lower Beurling density of its index set is zero.   
\end{abstract}

\maketitle
\section{Introduction}
Let $G$ be a second countable locally compact group and let $(\pi, \Hpi)$ be an irreducible, square-integrable projective unitary representation of $G$. For a vector $\eta \in \Hpi$ and a discrete subset $\Gamma \subseteq G$, a system of the form
\[
\pi(\Gamma) \eta = \big\{ \pi(\gamma) \eta : \gamma \in \Gamma \big\}
\]
is often referred to as a \emph{coherent state subsystem} or simply a \emph{coherent system}. 

A natural question on coherent systems that goes back to at least \cite{perelomov1972coherent}
is to derive quantitative necessary conditions on the index set $\Gamma \subseteq G$ for the system $\pi(\Gamma) \eta$ to have certain spanning/reproducing properties, such as completeness or orthogonality. For index sets forming a lattice $\Gamma$, necessary conditions on the lattice covolume $\covol(\Gamma)$ for $\pi(\Gamma) \eta$ being a complete system or orthogonal system have been obtained in, e.g., \cite{bekka2004square, romero2022density}. Similar results for frames and Riesz sequences associated to general index sets $\Gamma \subseteq G$ in a unimodular amenable group $G$ have been obtained in terms of the lower and upper Beurling densities $D^-(\Gamma)$ and $D^+(\Gamma)$ of $\Gamma$, respectively, in the papers \cite{papageorgiou2025counting, enstad2025coherent, caspers2023overcompleteness, fuehr2017density, mitkovski2020density, enstad2025dynamical}. In particular, the following necessary density condition for frames $\pi(\Gamma) \eta$ satisfying a mild localisation assumption $\eta \in \mathcal{B}_{\pi}$ (see \cite{grochenig2008homogeneous} and \Cref{bpi})
has been obtained via different methods in \cite{enstad2025coherent, papageorgiou2025counting, caspers2023overcompleteness, enstad2025dynamical}. See Section \ref{sec:prelim} for the precise definitions involved.

\begin{theorem} \label{thm:density}
    Let $G$ be a second countable unimodular amenable group. Let $(\pi, \Hpi)$ be an irreducible, square-integrable projective unitary representation of $G$ of formal degree $d_{\pi} > 0$. If there exist $\Gamma \subseteq G$ and $\eta \in \mathcal{B}_{\pi}$ such that $\pi (\Gamma) \eta$ is a frame for $\Hpi$, then
    \[
    D^- (\Gamma) \geq d_{\pi}. 
    \]
\end{theorem}

For abelian groups (more generally, IN groups), it is known that \Cref{thm:density} also holds without the localisation assumption $\eta \in \mathcal{B}_{\pi}$, see, e.g., \cite{ramanathan1995incompleteness, enstad2025dynamical}. 
It is an open question whether this assumption can be removed from \Cref{thm:density} in general; see  \cite[Question 10]{velthoven2024density}.

The aim of this note is to show the existence of coherent frames for which the lower Beurling density of their index set is zero. In particular, this shows that the localisation assumption cannot be removed from Theorem \ref{thm:density}. In addition, it shows that various results obtained in \cite{enstad2025coherent, papageorgiou2025counting, caspers2023overcompleteness, enstad2025dynamical}, such as the stability of frames under weak limits of translates \cite{enstad2025dynamical} or the frame measure formula \cite{caspers2023overcompleteness}, do not hold without the localisation assumption.

The following is the main result obtained in this note.

\begin{theorem} \label{thm:zerodensity_intro}
Let $(\pi, \Hpi)$ be an irreducible, square-integrable projective representation of a connected, simply connected unimodular solvable Lie group $G$ of exponential growth. Then there exists $\eta \in \Hpi$ and a set $\Gamma \subseteq G$ with $D^- (\Gamma) = 0$ such that $\pi(\Gamma) \eta$ is a frame for $\Hpi$.
\end{theorem}

We mention that the proof of \Cref{thm:zerodensity_intro} shows that the set $\Gamma \subseteq G$ can be chosen to be contained in a subgroup of $G$ of dimension at most $3$. 

Our proof of \Cref{thm:zerodensity_intro} is based on similar techniques as those used for the construction of frames of translates for nonnilpotent exponential solvable Lie groups in \cite{fuehr2022groups}. 
Specifically, as in \cite{fuehr2022groups}, we use that any such group admits a nonunimodular closed subgroup for which there exists a continuous frame of translates \cite{fuehr2002admissible} and combine this with the fact that any such continuous frame can be discretised to obtain a discrete frame \cite{freeman2019discretization}. We note, however, that \Cref{thm:zerodensity_intro} cannot  simply be obtained as a direct consequence (see, e.g., \cite[Thm. 2.11]{fuehr2022groups}) of the main results in \cite{fuehr2022groups} as these are only known for genuine regular representations (i.e., not for general projective regular representations) and since unimodular connected, simply connected solvable Lie groups do not admit irreducible (genuine) unitary representations that are square-integrable in the strict sense, cf. \cite[Cor. 4.2]{beltita2025square}. To circumvent these obstructions, we carefully use the correspondence between projective representations of a given group and representations of an associated central extension of the group and exploit the aforementioned results on frames of translates for the latter group.

The note is organised as follows. Section \ref{sec:prelim} provides the basic background on square-integrable projective representations, Beurling densities and frames. The proof of \Cref{thm:zerodensity_intro} is carried out in Section \ref{sec:proof}. Lastly, we provide in Section \ref{sec:examples}  a class of solvable Lie groups to which \Cref{thm:zerodensity_intro} applies. 

\section{Coherent systems over point sets in amenable groups} \label{sec:prelim}
Throughout this section, we let $G$ be a unimodular second countable group with Haar measure $\mu_G$. 

\subsection{Square-integrable projective representations} \label{sec:projective}
A projective unitary representation $(\pi, \Hpi)$ of $G$ on the separable Hilbert space $\Hc_\pi$ is a Borel map $\pi\colon G \to \mathcal{U}(\Hpi)$ satisfying $\pi(e) = I_{\Hpi}$ and such that
\[
\pi(x)\pi(y) = \sigma(x,y) \pi(xy) \quad \text{for all} \quad x, y \in G,
\]
for some Borel function $\sigma \colon G \times G \to \mathbb{T}$, called the \emph{cocycle} of $\pi$, that satisfies 
the cocycle conditions
\begin{eqnarray*}
    \sigma(x, e) =\sigma (e, x) =1 \quad \text{and} \quad
    \sigma(x, yz) \sigma(y, z) = \sigma (xy, z) \sigma(x, y), 
\end{eqnarray*}
for all $x, y, z\in G$. 

A projective unitary representation with cocycle $\sigma$ will also simply be referred to as a $\sigma$-representation. For any $\eta_1, \eta_2 \in \Hpi$, the absolute value $|\langle \eta_1, \pi(\cdot) \eta_2 \rangle | \colon G \to [0, \infty)$ is continuous by \cite[Lem. 7.1 and Thm. 7.5]{varadarajan1985geometry}. Moreover, since the unitary group of $\Hc_\pi$, with the weak (or, equivalently, strong) topology is a second countable topological group, any $1$-representation is automatically strongly continuous, cf. \cite[Lem. 5.28]{varadarajan1985geometry}. 

The projective unitary representation $(\pi, \Hpi)$ is said to be \emph{irreducible} if $\{0\}$ and $\Hpi$ are the only closed $\pi$-invariant subspaces. 
An irreducible projective representation $(\pi, \Hpi)$ is called \emph{square-integrable} if there exists $\eta \in \Hpi \setminus \{0\}$ such that
\begin{align} \label{eq:SI}
\int_G |\langle \eta, \pi(x) \eta \rangle |^2 \; d\mu_G (x) < \infty. 
\end{align}
In this case, there exists $d_{\pi} > 0$, called the \emph{formal degree} of $\pi$, such that
\[
\int_G |\langle f_1, \pi(x) f_2 \rangle |^2 \; d\mu_G (x) = d_{\pi}^{-1} \| f_1 \|^2 \| f_2 \|^2
\]
for all $f_1, f_2 \in \Hpi$; see, e.g., \cite[Thm. 2]{aniello2006square}. 

We mention that instead of working with $\sigma$-representations that are square-integrable in the strict sense of \Cref{eq:SI},  
one could alternatively work with $1$-representations that are square-integrable modulo the centre or projective kernel. Any such latter representation can, however, be treated
as a projective representation of the corresponding quotient quotient and is then square-integrable in the strict sense of
 \Cref{eq:SI}. Indeed, if $(\rho, \mathcal{H}_{\rho})$ is an irreducible $1$-representation of $G$ that is square-integrable modulo $Z(G)$ and $s : G/Z(G) \to G$ is a Borel cross-section, then $\pi := \rho \circ s$ is an irreducible, square-integrable projective representation of $G/Z(G)$,
 see, e.g., \cite[Prop. 3]{aniello2006square}.  
 It is worth noting that square-integrable $1$-representations do not exist for connected,
simply connected unimodular solvable Lie groups \cite[Cor. 4.2]{beltita2025square}.

\subsection{Beurling density}
A (right) \emph{strong F\o lner sequence} in $G$ is a sequence $(K_n)_{n \in \mathbb{N}}$ of nonnull compact sets $K_n \subseteq G$ satisfying, for any compact set $K \subseteq G$,
\[
\lim_{n \to \infty} \frac{\mu_G (K_n K \cap K_n^c K)}{\mu_G (K_n)} = 0. 
\]
A strong F\o lner sequence exists in any unimodular amenable group. Conversely, if $G$ is a second countable locally compact group admitting a strong F\o lner sequence, then it must be unimodular and amenable. See \cite[Prop. 11.1]{tessera2008large} for both claims; see also \cite[Prop. 5.10]{pogorzelski2022leptin}. We recall that any solvable locally compact group is amenable.

Given a discrete set $\Gamma \subseteq G$, its \emph{lower Beurling density} is defined by
\[
D^- (\Gamma) := \lim_{n \to \infty} \inf_{x \in G} \frac{\# (\Gamma \cap xK_n)}{\mu_G (K_n)},
\]
where $(K_n)_{n \in \mathbb{N}}$ is a strong F\o lner sequence in $G$. It is shown in \cite[Prop. 5.14]{pogorzelski2022leptin} that the Beurling density is independent of the choice of F\o lner sequence.

\subsection{Coherent frames}
A coherent system $\pi(\Gamma) \eta$ is said to be a \emph{frame} for $\Hpi$ if there exist $0<A\leq B < \infty$ such that
\[
A \| f \|^2 \leq \sum_{\gamma \in \Gamma} |\langle f, \pi(\gamma) \eta \rangle |^2 \leq B \| f \|^2 
\]
for all $f \in \Hpi$. 

A common condition in the study of coherent frames $\pi(\Gamma) \eta$ is that $\eta \in \Hpi$ belongs to the space
\begin{equation}\label{bpi}
\mathcal{B}_{\pi} := \bigg\{ \eta \in \Hpi : \int_G \sup_{q \in Q} |\langle f, \pi (x q) \eta \rangle |^2 \; d\mu_G (x) < \infty, \quad \text{for all} \quad f \in \Hpi \bigg\}, 
\end{equation}
where $Q$ is some fixed relatively compact unit neighbourhood, see, e.g., \cite{caspers2023overcompleteness, papageorgiou2025counting, enstad2025coherent, fuehr2017density, enstad2025dynamical}. The space $\mathcal{B}_{\pi}$ can be shown to be norm dense, see \cite{grochenig2008homogeneous}.
If $\pi(\Gamma) \eta$ is a frame for $\Hpi$ with $\eta \in \mathcal{B}_{\pi}$, then necessarily $G = \Gamma K$ for some compact set $K \subseteq G$, cf. \cite[Thm. 3.9]{enstad2025dynamical}.

\section{Proof of the main result}
\label{sec:proof}

This section is devoted to providing a proof of \Cref{thm:zerodensity_intro}. 
We start by recalling some preliminary results.

\subsection{Restriction of regular representations}
A key ingredient in proving \Cref{thm:zerodensity_intro} is given by the following theorem, cf.\cite[Proof of Thm. 3.5]{fuehr2022groups}. 
We provide another short proof for this result below.

\begin{theorem}[\cite{fuehr2022groups}] \label{thm:multiplicity}
    Let $G$ be a second countable locally compact group.
    Let $H \leq G$ be a closed subgroup such that its regular representation $\lambda_H$ has infinite multiplicities.
    Then $\lambda_H$ is unitarily equivalent to the restriction $\lambda_G|_{H}$ of the regular representation $\lambda_G$ of $G$.
\end{theorem}

For the proof, we recall that for a locally compact second countable group $G$ and a representation 
$\pi$ of $G$, the associated \emph{Fourier space} is defined by 
\[ A_\pi (G) := \overline{\Span }\big\{ \langle \eta_1, \pi (\cdot) \eta_2 \rangle : \eta_1, \eta_2 \in \Hpi \big\} \subseteq (C^*(G))^*, \]
where $(C^*(G))^*$ is endowed with the dual topology. 
Note that $A_\pi(G) \subseteq C_b(G)$. 
For the particular case that $\pi=\lambda_G$ is the regular representation, the space $A_{\lambda_G}$ is the \emph{Fourier algebra} of $G$ and simply denoted by $A(G)$. See \cite{Arsac, kaniuth2018fourier} for background.

\begin{proof}[Proof of \Cref{thm:multiplicity}]
A classical result of C.~Herz (see \cite[Prop.~3.23 and its proof]{Arsac}) says that 
$$ A(G)\vert_H=A_{\lambda\vert_H} = A(H).$$
It follows from \cite[Prop.~3.1]{Arsac} that $\lambda_G\vert_H$ and $\lambda_H$ are quasi-equivalent representations. See also \cite{kaniuth2018fourier} for both facts. 
By \cite[Prop.~5.3.1]{Dix77}, the quasi-equivalence of $\lambda_G|_H$ and $\lambda_H$ is equivalent to the existence of $N, M\in \{1, 2, \dots, \aleph_0\}$ such that 
$$ N\cdot \lambda_H \simeq M\cdot (\lambda_G\vert_ H).$$
On the other hand, since $\lambda_H$ has infinite multiplicity, i.e., $\lambda_H \simeq \infty \cdot \lambda_H$,
it follows that $ N\cdot \lambda_H  \simeq  N\cdot (\infty\cdot  \lambda_H) \simeq \infty\cdot  \lambda_H \simeq \lambda_H$.
Moreover, since
$\lambda_G \simeq\text{ind}_H^G (\lambda_H)$, it follows that $\lambda_G\vert_H$ has infinite multiplicity as well, 
thus $\lambda_G\vert_H\simeq M\cdot (\lambda_G\vert_ H)$ by the same argument as above. 
This yields that 
$$ \lambda_H\simeq  N\cdot \lambda_H \simeq M\cdot (\lambda_G\vert_H) \simeq \lambda_G\vert_H,$$
which settles the claim.
\end{proof}

\subsection{Solvable Lie groups of exponential growth}
We first recall two important examples of nonunimodular exponential Lie groups that play 
an important role in our proof. A further key property of these examples is that they are both absolutely closed groups, that is, for any imbedding (continuous injective homomorphism) in a topological group, the image is a closed subgroup, or equivalently, any imbedding in a Lie group is open onto its image, see \cite[Thm.]{Goto73}.

\begin{example}\label{ex-groups}\hfill
    \normalfont
    
(a) The \emph{ affine group} (or the $ax+b$ \emph{group}) is the semidirect product ${\mathbb R}  \rtimes_\alpha  {\mathbb R} $ where
the action
$ \alpha \colon  {\mathbb R}\times  {\mathbb R} \to  {\mathbb R} $ is given by 
$$  \alpha (x, y)  = e^{- x} y, \quad x\in {\mathbb R}, y\in {\mathbb R}. $$
It is a nonunimodular exponential Lie group with Lie algebra generated by a basis $X, Y \in {\mathbb R}^2$ 
with nontrivial bracket 
 $$ [X, Y]=Y;$$
see also \cite[\S 4.3]{fujludw}.
This is an absolutely closed group by \cite[Thm.~1.2(a)]{Omori66}.

(b) The \emph{Gr\'elaud group} $G_\theta$, $\theta\in {\mathbb R}\setminus \{0\}$, is the semidirect product 
$ {\mathbb R} \rtimes_{\alpha_{D_\theta}} {\mathbb R}^2 $ where the action
$ \alpha_{D_\theta} \colon  {\mathbb R}\times  {\mathbb R}^2 \to  {\mathbb R}^2$ is given by 
$$  \alpha_{D_\theta} (t, x)  = e^{- t D_\theta} x, \quad t\in {\mathbb R}, x\in {\mathbb R}^2, $$
where $D_{\theta}=\left (\begin{matrix} 1  & \theta\\
-\theta & 1\end{matrix}\right )$. 
It is a nonunimodular exponential Lie group with Lie algebra generated by a basis $X, Y_1, Y_2 \in {\mathbb R}^3$ 
with nontrivial brackets
$$ [X, Y_1] = Y_1 -\theta Y_2, \; [X, Y_2] = Y_2 + \theta Y_1.$$
See also \cite[\S 4.4]{fujludw}.
This is an absolutely closed group by \cite[Thm.~1.2(c)]{Omori66}.
\end{example}

We recall that a connected Lie group $G$ has exponential growth if for every compact unit neighbourhood $U \subseteq G$, there exists $t > 1$ such that 
\[ 
\mu_G (U^n) \geq t^n
\]
for 
$ n \in \mathbb{N}$.

\begin{remark}\label{jenk}
By \cite[Thm. 1.4, Cor. 2.1]{jenkins1973growth}, a connected Lie group $G$ is of exponential growth if and only if its Lie algebra $\mathfrak{g}$ is not of type R, 
that is, there exists $X \in \mathfrak{g}$ such that spectrum of the linear map $\ad(X)$ is not contained in $ i \mathbb{R}$.
The Lie algebra $\mathfrak{g}$ of such a group has a subalgebra
 that is isomorphic to the Lie algebra of either the affine or the Gr\'elaud group by
     \cite[Lem.~1.6]{jenkins1973growth}.
    \end{remark}

\begin{lemma}\label{subg_lemma}
Let $S$ be a
connected solvable Lie group of exponential growth. 
Then there exists a closed, connected and simply connected subgroup $H$ of $S$ that is isomorphic to either the affine group or the Gr\'elaud group. 
\end{lemma}

\begin{proof} 
Assume first that $S$ is simply connected with Lie algebra $\mathfrak{s}$.  Let $\mathfrak{h}$  be the Lie subalgebra $\mathfrak{h}$ which is isomorphic to either the affine Lie algebra or the Gr\' elaud's Lie algebra, cf. Remark~\ref{jenk}. If $H$ is the unique connected Lie subgroup of $S$ with Lie algebra $\mathfrak{h}$, then it is automatically closed and simply connected, see, e.g., \cite[Prop. 11.2.15]{hilgert2012structure}. 

If $S$ is not simply connected, 
then we let $\widetilde{S}$ be the universal cover of $S$. 
We denote by $p  \colon \widetilde{S}\to S$ the covering map, so that 
$\ker(p)$ is a discrete central subgroup of $\widetilde{S}$.  
The Lie algebra $\mathfrak{s}$ of the simply connected group $\widetilde{S}$ is the same as the Lie algebra of $S$.
Since $S$ has exponential growth,  $\mathfrak{s}$ admits a Lie subalgebra $\mathfrak{h}$ which is isomorphic to either the affine Lie algebra or Gr\' elaud's Lie algebra, cf. Remark~\ref{jenk}. 
Denote by $\widetilde{H}$ the unique closed simply connected Lie subgroup of $\widetilde{S}$ with Lie algebra $\mathfrak{h}$.
Since $\widetilde{H}$ has trivial centre,  it follows that $\widetilde{H}\cap \ker(p)=\{e_{\widetilde{S}}\}$. 
Hence, the continuous map $p\vert_{\widetilde{H}} \colon \widetilde{H}\to S$ is injective.
The group $\widetilde{H}$ is isomorphic either with the affine group or the Gr\'elaud group, thus absolutely closed, hence $H:=p(\widetilde{H})$ is a closed, connected and simply connected subgroup of $S$ which is also isomorphic either to the affine group or the Gr\'elaud group. 
\end{proof}

\subsection{Proof of Theorem \ref{thm:zerodensity_intro}}
Let $\sigma : G \times G \to \mathbb{T}$ be the cocycle of $(\pi, \Hpi)$.

We first assume that $\sigma$ is an analytic cocycle. 
Define the group $G_{\sigma} := G \times \mathbb{T}$ with group law $(x, u) (y, v) = (xy, uv \sigma(x,y))$. 
By \cite[Thm. 7.8 and Thm. 7.21]{varadarajan1985geometry}, the group $G_{\sigma}$ is a Lie group with respect to the product topology and such that the quotient map $j \colon G_{\sigma} \to G, \; (x,u) \mapsto x$ is a smooth group homomorphism. Moreover, it is connected and it is solvable as a central extension of a solvable group. 
The product measure $\mu_G \otimes \mu_{\mathbb{T}}$ forms a left Haar measure on $G_{\sigma}$, and its modular function is given by $\Delta_{G_{\sigma}} = \Delta_{G} = 1$. Since $G$ has exponential growth, it follows that also $G_{\sigma}$ has exponential growth.  

Let $H$ be the closed simply connected subgroup of $G_\sigma$ that is isomorphic with either the affine group or Gr\'elaud's group, cf. Lemma~\ref{subg_lemma}. In particular, we must have that $H \subseteq G \times \{1\}$ as $H$ is simply connected.
Since $H$ is nonunimodular and exponential (hence, type I), there exists $F_0 \in L^2 (H)$ such that $\lambda_{H} (H) F_0$ is a continuous Parseval frame for $L^2 (H)$, that is,  
\[
\| F \|^2 = \int_{H} |\langle F, \lambda_{H} (h) F_0 \rangle |^2 \; d\mu_{H} (h), \quad \text{for all} \quad F \in L^2 (H),
\] 
cf. \cite[Thm. 0.2]{fuehr2002admissible}. 
Hence, by \cite[Thm. 1.3]{freeman2019discretization} or \cite[Thm. 3.4]{bownik2026uniform}, there exists a  discrete subset $\Lambda \subseteq H$ such that $\lambda_{H} (\Lambda) F_0$ is a frame for $L^2 (H)$. See also \cite[Thm. 3.2]{fuehr2022groups}.
Moreover, since $\lambda_{H}$ has infinite multiplicities by \cite[Cor. 3.51]{fuehr2005abstract}, 
it follows from \Cref{thm:multiplicity} that there 
exists a unitary operator $U \colon L^2(H) \to L^2(G_{\sigma})$  such that 
$\lambda_{H} (x) = U^{-1} \circ \lambda_{G_{\sigma}} (x) \circ U$ for every $x\in H$.
Setting $F_0' := UF_0$,  it follows that $\lambda_{G_{\sigma}} (\Lambda) F_0'$ is a frame for $L^2 (G_{\sigma})$.  

Consider next the $1$-representation $(\pi_{\sigma}, \Hpi)$ of $G_{\sigma}$ defined by $\pi_{\sigma} (x, u) = u \pi(x)$. Then $\pi_{\sigma}$ is strongly continuous, irreducible and square-integrable. Hence, there exists $\varphi \in \Hpi \setminus \{0\}$ such that the coefficient map $C_{\varphi} : \Hpi \to L^2 (G_{\sigma}), \; C_{\varphi} f = \langle f, \pi_{\sigma} (\cdot) \varphi \rangle$ is an isometry. Moreover, we have that
\[
C_{\varphi} \pi_{\sigma} (x,u) f = \lambda_{G_{\sigma}} (x,u) C_{\varphi} f, \quad (x,u) \in G_{\sigma}. 
\]
In particular, the image space $C_{\varphi} (\Hpi)$ is a closed, left-invariant subspace of $L^2(G_{\sigma})$ and the orthogonal projection $P$ from $L^2 (G_{\sigma})$ onto $C_{\varphi} (\Hpi)$ is given by $P = C_{\varphi} C_{\varphi}^*$, see, e.g., \cite[Prop. 2.12]{fuehr2005abstract}. Note that $\lambda_{G_{\sigma}} (\Lambda) P F_0' = P \lambda_{G_{\sigma}} (\Lambda) F_0'$ is a frame for $C_{\varphi} (\Hpi)$. As such, the system 
\[
\pi_{\sigma} (\Lambda) C_{\varphi}^* P F_0' = C_{\varphi}^* \lambda_{G_{\sigma}} (\Lambda) P F_0'
\]
forms a frame for $\Hpi$.

Define $\eta := C_{\varphi}^* P F_0' \in \Hpi$ and $\Gamma := j (\Lambda) \subseteq G$, where $j : G_{\sigma} \to G$ is the smooth group homomorphism given by $(x,u) \mapsto x$.
Since $\Lambda \subseteq H \subseteq G \times \{1\}$, 
it follows that
\[
\pi(\Gamma) \eta = \pi_{\sigma} (\Lambda) \eta,
\]
which is a frame for $\Hpi$. In particular, it follows from \cite[Lem. 3.2]{papageorgiou2025counting} that $\Gamma$ must be relatively separated, i.e., $\sup_{x \in G} \# (\Gamma \cap x K) < \infty$ for any compact set $K \subseteq G$.

In order to show that $D^-(\Gamma) = 0$, note first that $H \cap \ker(j) = \{e_{G_{\sigma}} \}$. Therefore, the continuous map  
$j\vert_{H} \colon H \to G$ is injective. Since $H$ is  isomorphic either with the affine group or the Grélaud group, 
it follows that  $H' := j(H)$ is a closed subgroup of $G$ isomorphic with $H$. 
Since $H'$ is nonunimodular, we have that $H' \subsetneq G$ and thus $G / H' \cong \mathbb{R}^{\dim(G/H')}$ is noncompact, cf.  \cite[Prop. 11.2.15]{hilgert2012structure}.
Arguing by contradiction, assume that  $ D^-(\Gamma) >  0$. Let $(K_n)_{n \in \mathbb{N}}$ be a strong F\o lner sequence consisting of nonnull compact sets. Then there exists $N  \in \mathbb{N}$  such that, for all $n \geq N$ and $x \in G$, 
\[
\# (\Gamma \cap x K_n) > \frac{1}{2}D^-(\Gamma) \mu_G (K_n) > 0.
\]
Setting $K := K_N \subseteq G$, we obtain that $\# (\Gamma \cap x K) \geq 1$ for all $x \in G$. Note that this is equivalent to 
$G = \Gamma K^{-1}$. 
Since $\Gamma \subseteq H'$, we also have $G = H' K^{-1}$. If $p : G \to H' \backslash G$ denotes the canonical projection onto the right coset space, then it follows that  $H' \backslash G = p(G) = p(H' K^{-1}) =  p(K^{-1})$ is compact, which contradicts that $G/H'$ is not compact, since $G/H' \cong H' \backslash G$ as topological spaces.
Therefore, we must have that $D^-(\Gamma) = 0$. 

Lastly, we treat the case that $\sigma$ is not an analytic cocycle. By \cite[Cor. 7.30]{varadarajan1985geometry}, there exists an analytic cocycle $\sigma'$ that is similar to $\sigma$, in the sense that there exists a Borel measurable function
$a : G \to \mathbb{T}$ such that
\[
\sigma'(x,y) = \frac{a(xy)}{a(x)a(y)} \sigma(x,y).
\]
The operator $\pi'(x) = a(x) \pi(x)$ defines a $\sigma'$-representation $(\pi', \Hpi)$ of $G$. Hence, by the case of an analytic cocycle treated above, there exists $\Gamma \subseteq G$ with $D^-(\Gamma) =0$ and $\eta \in \Hpi$ such that $\pi'(\Gamma) \eta$ is a frame for $\Hpi$. 
Clearly,  $\pi(\Gamma) \eta$ is also a frame.

\qed

\section{A class of examples} \label{sec:examples}
In this section we construct a class of unimodular exponential Lie groups of exponential growth that admit square-integrable representations modulo the centre. 
In particular, these give rise to square-integrable projective representations of the corresponding quotient groups; see \Cref{sec:projective}.

 Let $N$ be a connected, simply connected nilpotent Lie group with Lie algebra ${\mathfrak n}$ of odd  dimension $n$. We assume that 
 the
 centre $Z(N)$ of $N$ has dimension $1$ and let $\{X_n, ..., X_1\}$ be a Jordan-H\"older basis for $\mathfrak{n}$ with ${\mathfrak z}({\mathfrak n}) ={\mathbb R} X_1$. We write $N=({\mathfrak n}, \cdot)$, where $\cdot$ is the Baker-Campbell-Hausdorff multiplication, so that the identity element of $N$ is $0$ and the inverse of $x$ is $-x$.
 Let $\ell_{\mathfrak n}\in {\mathfrak z}({\mathfrak n})^*$ be such that $\ell_{\mathfrak n}(X_1)=1$  
and assume that the stabiliser
\[ {\mathfrak n}(\ell_{\mathfrak n}) := \big\{ X \in \mathfrak{n} : \ell_{\mathfrak{n}} ([X,Y]) = 0, \; \forall Y \in \mathfrak{n} \big\} = {\mathfrak z}({\mathfrak n}). \]

Consider the connected, simply connected nilpotent Lie group $G_0:= {\mathbb R} \times N$ with Lie algebra ${\mathfrak g}_0={\mathbb R}\times  {\mathfrak n} =  {\mathbb R} \dot{+}{\mathfrak n}$.
 Then ${\mathfrak g}_0$ has even dimension $n+1$. 
Let 
$X_{n+1} \in \mathfrak{g}_0$ be such that it generates ${\mathbb R}\times \{0\}\subseteq \mathfrak{g}_0$. 
Using the chart 
$$ (x_{n+1}, \dots , x_1) \mapsto \exp_{G_0}\bigg( \sum_{j=1}^{n+1} x_j X_j \bigg)$$
to describe $G_0$, a point $x \in G_0$
has the form 
$ x= (x_{n+1}, \dots , x_1):= \sum\limits_{j=1}^{n+1} x_j X_j$. Here, we recall that we use the Baker-Campbell-Hausdorff multiplication and identify the group, as a set, with its Lie algebra.
Note that the centre ${\mathfrak z}({\mathfrak g}_0) ={\mathbb R} \dot{+} {\mathfrak z} ({\mathfrak n})=\text{span}\,\{X_{n+1}, X_1\}$ is 2-dimensional.
We have that 
\[ {\mathfrak g}_0(\ell_0) := \big\{ X \in \mathfrak{g}_0 : \ell_0 ([X,Y]) = 0, \; \forall Y \in \mathfrak{g}_0 \big\} = {\mathfrak z}({\mathfrak g}_0), \]
where $\ell_0$ is the extension of $\ell_ {\mathfrak n}$  to ${\mathfrak g}_0$ with $\ell_0(X_{n+1})=0$.

Assume that there is $D_d\in \text{Der}({\mathfrak g}_0)$  of the form  $D_d=\text{diag}\,\{0,  \lambda_n, \dots, \lambda_2, 0\}$, $\lambda_j \in {\mathbb R}$, $\lambda_j \ne 0$, $j=2, \dots, n$ and such that $\sum_{j=2}^n \lambda_j =\text{tr} ( D_d)=0$.
Let $E$ 
be the nilpotent map such that $E X_{n+1}= X_{1}$  and $E X_j=0$ for 
$j\in \{ 1, \dots, n\}$. 
Then $ED_d= D_dE=0$ and $E^2=0$. 
Define
$$ D= D_d+E.$$
Consider the action 
$ \alpha\colon {\mathbb R} \to \Aut(G_0)$
of ${\mathbb R}$ on $G_0$ defined by
$$ \alpha_t(x) =\ee^{tD}x, \quad x\in G_0, \, t\in {\mathbb R}.$$
Let $G$ be the semidirect product $G=G_0 \rtimes_\alpha {\mathbb R}$, with the multiplication given by 
$$ (x, t)(y, s) = (x\cdot (\ee^{tD} y), t+s), \quad x, y\in G_0, \, s, t \in {\mathbb R}.$$
The Lie algebra $\mathfrak{g}$ of $G$ is $\mathfrak{g}_0\times {\mathbb R}$ with Lie bracket 
$$ [(X, t), (Y, s)]= (tDY- sDX+[X, Y], 0), \quad (X, t), (Y, s) \in \mathfrak{g}_0\times {\mathbb R}.$$
Since $E$ is nilpotent, it follows that the spectrum of  $\mathrm{ad}_{\mathfrak{g}}(X, t)$, $(X,t)\in  \mathfrak g$, is precisely $t\sigma(D_d)$, that is, $\{0, t\lambda_2, \dots, t\lambda_n\}\subseteq \mathbb{R}$.  
Hence, $G$ is an exponential Lie group (even completely solvable). 
Since there is a real eigenvalue $\lambda_j \ne 0$, the group $G$ is not of type $R$, hence it is of exponential growth, cf. \Cref{jenk}. 
Moreover, the modular function of $G$ is 
$\Delta_G= e^{-\mathrm{tr(D_d)}} =1$, thus $G$ is also unimodular. 
The form of the derivation $D$ ensures that the
centre of ${\mathfrak g}$ is given by ${\mathfrak z}({\mathfrak g})={\mathbb R} X_1$.

If $\ell\in {\mathfrak g}^*$ is an extension of $\ell_0$, 
then
${\mathfrak g}(\ell) ={\mathfrak z}({\mathfrak g})$, so that the irreducible representation 
$\pi$ of $G$
corresponding to the coadjoint orbit $G \cdot \ell =\Ad^*(G)\ell$ is square-integrable modulo $Z(G)$,
see, e.g., \cite[\S 4.1]{rosenberg1978square}. 

In the special case that $\mathfrak{n}$ is the $3$-dimensional Heisenberg Lie algebra and the derivation $D_d = \text{diag}\{ 0, 1, -1, 0 \}$, the above construction recovers the example \cite[\S 4.13]{rosenberg1978square}.

\section*{Acknowledgements}
For J. v. V., this research was funded in whole or in part by the
Austrian Science Fund (FWF): 10.55776/PAT2545623. 

For I.B. this research was supported by a grant of the Ministry of Research, Innovation and Digitization,
CNCS–UEFISCDI, project number PN-IV-P1-PCE-2023-0264, within PNCDI IV.

\bibliographystyle{abbrv}
\bibliography{bib}

\end{document}